\documentclass[final]{amsart}
\usepackage{multirow}
\usepackage{amsmath}
\usepackage{amssymb}
\usepackage{amsmath, amssymb, amsfonts, amsxtra}
\theoremstyle{proclaim}
\newtheorem{theorem}{Theorem}[section]
\newtheorem{lemma}[theorem]{Lemma}
\newtheorem{corollary}[theorem]{Corollary}

\theoremstyle{fancyproclaim}

\theoremstyle{statement}
\newtheorem{remark}[theorem]{Remark}
\newtheorem{definition}[theorem]{Definition}

\theoremstyle{fancystatement}

\numberwithin{equation}{section}

\providecommand{\AMS}{$\mathcal{A}$\kern-.1667em%
	\lower.25em\hbox{$\mathcal{M}$}\kern-.125em$\mathcal{S}$}
\usepackage{xcolor}
\begin{document}
	
		\title[$k-$smoothness on polyhedral Banach spaces ]{ $k-$smoothness on polyhedral Banach spaces}
		\author[ Subhrajit Dey, Arpita Mal and Kallol Paul]{ Subhrajit Dey, Arpita Mal and Kallol Paul}

		\newcommand{\acr}{\newline\indent}

		\address[Dey]{Department of Mathematics\\ Muralidhar Girls' College\\ Kolkata 700029\\ West Bengal\\ INDIA}
		\email{subhrajitdeyjumath@gmail.com}
		
			\address[Mal]{Department of Mathematics\\ Jadavpur University\\ Kolkata 700032\\ West Bengal\\ INDIA}
		\email{arpitamalju@gmail.com}

		\address[Paul]{Department of Mathematics\\ Jadavpur University\\ Kolkata 700032\\ West Bengal\\ INDIA}
		\email{kalloldada@gmail.com}

		\thanks{The research of  Arpita Mal is supported by UGC, Govt. of India.  The research of Prof. Kallol Paul  is supported by project MATRICS (MTR/2017/000059)  of DST, Govt. of India. }


		\maketitle
		\begin{abstract}
			We characterize $k-$smoothness of an element on the unit sphere of a finite-dimensional polyhedral Banach space. Then we study $k-$smoothness of an operator $T \in \mathbb{L}(\ell_{\infty}^n,\mathbb{Y}),$ where $\mathbb{Y}$ is a two-dimensional Banach space with the additional condition that $T$ attains norm at each extreme point of $B_{\ell_{\infty}^{n}}.$ We also characterize $k-$smoothness of an operator defined between $\ell_{\infty}^3$ and $\ell_{1}^3.$
		\end{abstract}

		\section{Introduction}
		The study of $k-$smoothness plays an important role to identify the structure of the unit ball of a Banach space. The papers \cite{H,Ha,KS,LR} contain the study of $k-$smooth points of many of the Banach spaces. There are several papers including \cite{H,KS,LR,MDP,MP,MPD,Wa} that contain the study of $k-$smoothness of operators on different spaces. In \cite{MPD}, authors have obtained a relation between $k-$smoothness and extreme points of the unit ball of a polyhedral Banach space. The purpose of this paper is to characterize the order of smoothness of an element on the unit sphere of a finite-dimensional polyhedral Banach space, we also study $k$-smoothness of an operator defined between polyhedral Banach spaces. Let us  first fix the  notation and terminology.
		
		Letters $\mathbb{X},$ $\mathbb{Y}$ denote Banach spaces. Throughout the paper we assume the Banach spaces to be real. We denote the unit ball and the unit sphere of $\mathbb{X}$ respectively by $B_\mathbb{X}$ and $S_\mathbb{X}$, i.e., $B_{\mathbb{X}}=\{x\in \mathbb{X}:\|x\|\leq 1\},S_\mathbb{X}=\{x\in \mathbb{X}:\|x\|= 1\}.$ Let $\mathbb{L}(\mathbb{X},\mathbb{Y})$ denote the space of all bounded linear operators between  $\mathbb{X}$ and  $\mathbb{Y}.$  For $T\in \mathbb{L}(\mathbb{X},\mathbb{Y}),~M_T$ denotes the collection of all unit vectors of $\mathbb{X}$ at which $T$ attains its norm, i.e., $M_T=\{x\in S_\mathbb{X}:\|Tx\|=\|T\|\}$. For a set $A,$ the cardinality of $A$ is denoted by $|A|.$ The dual space of $\mathbb{X}$ is denoted by $\mathbb{X}^*.$ An element $x \in S_{\mathbb{X}}$ is said to be an extreme point of the convex set $ B_{\mathbb{X}}$ if and only if $ x = (1-t)y + tz$ for some $ y,z \in B_{\mathbb{X}}$ and $ t \in (0,1) $ implies that $y=z=x.$ For $x,y\in \mathbb{X},$ let  $L[x,y]=\{tx+(1-t)y:0\leq t\leq 1\}$ and $L(x,y)=\{tx+(1-t)y:0< t< 1\}.$ The set of all extreme points of  $ B_{\mathbb{X}}$ is denoted by $Ext(B_\mathbb{X}).$  An element $x^*\in S_{\mathbb{X}^*}$ is  said to be a supporting linear functional of $x\in S_{\mathbb{X}},$ if $x^*(x)=1.$ For a unit vector $x,$ let $J(x)$ denote the set of all supporting linear functionals of $x,$ i.e., $J(x)=\{x^*\in S_{\mathbb{X}^*}:x^*(x)=1\}.$ The set $J(x)$ for $x \in S_{\mathbb{X}}$ plays a significant role to study the $k-$smoothness. 
		By the Hahn-Banach Theorem, it is easy to verify that $J(x)\neq \emptyset,$ for all $x\in S_\mathbb{X}.$ We would like to mention that $J(x)$ is a weak*-compact convex subset of $S_{\mathbb{X}^*}.$  A unit vector $x$ is said to be a smooth point if $J(x)$ is singleton.  $\mathbb{X}$ is said to be a smooth Banach space if every unit vector of $\mathbb{X}$ is smooth.  The set of all extreme points of $J(x)$ is denoted by $Ext~J(x),$ where $x\in S_\mathbb{X}.$ In 2005, Khalil and Saleh \cite{KS} defined $k-$smooth points as follows: An element $x\in S_{\mathbb{X}}$ is said to be $k-$smooth or the order of smoothness of $x$ is $k,$ if $J(x)$ contains exactly $k$ linearly independent supporting linear functionals of $x$. In other words, $x$ is $k-$smooth, if $\dim~span~J(x)=k.$ Moreover, from \cite[Prop. 2.1]{LR}, we get that $x$ is $k-$smooth, if $k=\dim~span~Ext~J(x).$   Similarly, $T\in \mathbb{L}(\mathbb{X},\mathbb{Y})$ is said to be $k-$smooth operator, if $k=\dim~span~J(T)=\dim~span~Ext~J(T).$ Observe that, $1-$smooth points of $S_\mathbb{X}$ are  the smooth points of $S_\mathbb{X}.$    The spaces that we are dealing with in this paper  are mostly finite-dimensional polyhedral Banach spaces, so it is appropriate to  state the following definitions.
		
		\begin{definition}
			A polyhedron $P$ is a non-empty compact subset of $\mathbb{X}$ which is the intersection of finitely many closed half-spaces of $\mathbb{X},$ i.e., $P=\cap_{i=1}^{r}M_i,$ where $M_i$ are closed half-spaces in $\mathbb{X}$ and $r\in \mathbb{N}.$ The dimension $\dim(P)$ of the polyhedron $P$ is defined as the dimension of the subspace generated by the differences $v-w$ of vectors $v,w\in P.$
		\end{definition}
		
		\begin{definition}
			A polyhedron $Q$ is said to be a face of the polyhedron $P$ if either $Q=P$ or if we can write $Q=P\cap \delta M,$ where $M$ is a closed half-space in $\mathbb{X}$ containing $P$ and $\delta M$ denotes the boundary of $M.$ If $\dim(Q)=i,$ then $Q$ is called an $i-$face of $P$. If $\dim(P)=n,$ then $(n-1)-$faces of $P$ are called facets of $P$ and $1-$faces of $P$ are called edges of $P.$
		\end{definition}     
	
	\begin{definition}
		 A finite-dimensional Banach space $\mathbb{X}$ is said to be polyhedral  if the unit ball $B_\mathbb{X}$ of $\mathbb{X}$ contains only finitely many extreme points.  Equivalently, a finite-dimensional Banach space $\mathbb{X}$ is a polyhedral Banach space, if  $B_\mathbb{X}$ is a polyhedron.  In particular, a two-dimensional polyhedral Banach space is said to be a polygonal Banach space. 
	\end{definition}

For a convex set $C,~\text{int}_r(C)$ denotes the relative interior of the set $C,$ i.e., $x\in \text{int}_r(C)$ if there exists $\epsilon>0$ such that $B(x,\epsilon)\cap \text{affine}(C)\subseteq C,$ where $\text{affine}(C)$ is the intersection of all affine sets containing $C$ and an affine set is defined as the translation of a vector subspace. A non-empty convex subset $F$ of $C$ is said to be a face of $C,$ if for $x,y\in C$ and $t\in (0,1),$ $(1-t)x+ty\in F\Rightarrow x,y\in F.$
		\smallskip
		
		In this paper, we first prove that a point on the relative interior of an $i$-face of the unit ball of an $n-$dimensional polyhedral Banach space is $(n-i)-$smooth. In \cite{MP}, the authors completely characterized the $k-$smoothness of an operator defined between two Banach spaces $\mathbb{X}$ and $\mathbb{Y},$ where $\dim(\mathbb{X})=\dim(\mathbb{Y})=2$ and in \cite{MDP}, the authors characterized the $k-$smoothness of a bounded linear operator defined between $\ell_{\infty}^{3}$ and a two-dimensional Banach space. We continue our study in this direction and characterize the $k-$smoothness of a bounded linear operator defined between $\ell_{\infty}^{n}$ and a two-dimensional Banach space with the assumption that the linear operator attains its norm at all the extreme points of the unit ball of $\ell_{\infty}^{n}.$ Then we characterize $k-$smoothness of a bounded linear operator defined between  $\ell_{\infty}^{3}$ and $\ell_{1}^{3}.$ 
		
		We state the following lemma \cite[Lemma 3.1]{W}, characterizing $Ext~J(T),$ which will be used often. For simplicity we state the lemma for finite-dimensional Banach spaces.
		\begin{lemma}\cite[Lemma 3.1]{W}\label{lemma-wojcik}
			Suppose that $\mathbb{X},\mathbb{Y}$ are finite-dimensional Banach spaces.  Let $T\in \mathbb{L}(\mathbb{X},\mathbb{Y})$ and $\|T\|=1$ Then 
			\[Ext ~J(T)=\{y^*\otimes x\in \mathbb{L}(\mathbb{X},\mathbb{Y})^*:x\in M_T\cap Ext(B_{\mathbb{X}}), y^*\in Ext ~J(Tx)\},\]
			where  $y^*\otimes x: \mathbb{L}(\mathbb{X},\mathbb{Y})\to \mathbb{R}$ is defined by $y^*\otimes x(S)=y^*(Sx)$ for every $S\in \mathbb{L}(\mathbb{X},\mathbb{Y}).$
		\end{lemma}

		\section{Main results}
		We begin this section with a relation between the order of smoothness of a unit vector $x$ in a polyhedral Banach space and the dimension of the face $F$ such that $x$ is in the relative interior of $F.$ 

		\begin{theorem}\label{th-poly}
			Let $\mathbb{X}$ be an $n$-dimensional polyhedral Banach space. Let $F$ be an $i$-face of $B_{\mathbb{X}}.$ Let $x \in \text{int}_r(F).$ Then $x$ is $(n-i)-$smooth.
		\end{theorem}
		
		\begin{proof}
			Let $f \in Ext~J(x).$ Since $x \in \text{int}_r(F),$ we have for all $y \in F,$ $f(y)=1.$ Therefore,
			 $$F \subseteq \cap_{f \in Ext~J(x)} \{y \in S_{\mathbb{X}} : f(y)=1 \}=A~\text{(say)}.$$ 
			 Clearly, $A$ is a face of $B_{\mathbb{X}}.$ If possible, suppose that $F \subsetneqq A.$ Then there exists $z \in A\setminus F.$ Now, $x \in F \subseteq A$ and $z \in A \Rightarrow tx+(1-t)z \in A$ for all $t \in [0,1],$ since $A$ is a face. Using convexity argument of norm, it is easy to observe that $\|x+\lambda(z-x)\|\geq \|x\|$ for all scalars $\lambda.$ Moreover, we have $\|x + (z-x)\|=\|z\| =1.$ If possible, let $\|x + \lambda_0 (z-x) \|=1$ for some $\lambda_0 < 0.$ Then
			  $$x =tz+(1-t)\{x+\lambda_0(z-x)\},~\text{ where~} t=\frac{-\lambda_0}{1-\lambda_0}\in (0,1).$$
			Since  $F$ is a face and $x \in F,$ we get $ z \in F,$ a contradiction. Thus $\|x + \lambda (z-x)\|>1$ for all $\lambda < 0.$ Let $Y=span\{x, z \}.$ Then $\dim(Y)=2$ and $Y$ is a polygonal Banach space. Using $\|x+(z-x)\|=1$ and $\|x+\lambda(z-x)\|>1$ for all $\lambda<0,$ it is easy to observe that $x \in Ext(B_{Y}).$ Thus, by \cite[Th. 3.5]{MP}, we get $x$ is $2-$smooth in $Y.$ Let $h,g$ be two linearly independent elements of $Ext(B_{Y^*})$ such that $h(x)=g(x)=1.$ If $h(z)=g(z)=1,$ then $h(x-z)=g(x-z)=0 \Rightarrow \ker(h)=\ker(g) \Rightarrow h=\lambda g,$ for some scalar $\lambda,$ a contradiction. Without loss of generality, suppose that $g(z) \neq 1.$ By \cite[Lemma 1.2, Page 168]{S}, there exists $g_1 \in Ext(B_{\mathbb{X}^*})$ such that $g_1|_{Y}=g.$ Now, $g_1(x)=1$ and $g_1 \in Ext(B_{\mathbb{X}^*}) \Rightarrow g_1 \in Ext~J(x).$ Thus, $g_1(z) \neq 1$ contradicts that $z \in A.$ Therefore, $A \setminus F = \emptyset \Rightarrow A=F,$ i.e., 
			$$F=\cap_{f \in Ext~J(x)} \{y \in S_{\mathbb{X}} : f(y)=1 \}=\cap_{f \in Ext~J(x)} \left(x + \ker(f) \right) \cap S_{\mathbb{X}}.$$
			This implies that $i=\dim(F)=\dim(\cap_{f \in Ext~J(x)} ker(f)).$ Now, let $x$ be $k-$smooth. Let $\{f_1, f_2,\dots, f_k \}$ be the set of all linearly independent vectors of $Ext~J(x).$ Then 
			\begin{eqnarray*}
				&& f \in Ext~J(x) \Rightarrow f=\sum_{j=1}^{k} a_{j} f_{j},~(a_j\in \mathbb{R})\\
				&\Rightarrow& \cap_{j=1}^{k} \ker(f_{j}) \subseteq \ker(f)\\
				&\Rightarrow& \cap_{j=1}^{k} \ker(f_{j}) \subseteq \cap_{f \in Ext~J(x)} ker(f) \subseteq \cap_{j=1}^{k} \ker(f_{j})\\
				&\Rightarrow& \cap_{f \in Ext~J(x)} \ker(f) = \cap_{j=1}^{k} \ker(f_{j})\\
				&\Rightarrow& i=\dim(\cap_{j=1}^{k} \ker(f_{j}))=n-k\\
				&\Rightarrow& k=n-i.
			\end{eqnarray*} 
			This completes the proof of the theorem.
		\end{proof}
		
	\begin{remark}
		Note that, if $\mathbb{X}$ is an $n$-dimensional polyhedral Banach space and $F$ is a facet of $B_\mathbb{X},$ then from Theorem \ref{th-poly}, we get for each $x\in int_r(F),$ $x$ is smooth. On the other hand, if $F$ is a $0$-face, i.e., $F=\{x\},$ then $x$ is $n-$smooth. In this case, clearly $x$ is an extreme point of $B_\mathbb{X}.$ It is worth mentioning that Theorem \ref{th-poly} generalizes \cite[Th. 3.5]{MP}.  
	\end{remark}

		Now, we focus on  the space of all operators defined between some particular polyhedral Banach spaces. First we study $k-$smoothness of an operator defined between $\ell_{\infty}^{n}$ and an arbitrary two-dimensional Banach space. To do so we need the following two lemmas. 
		
		\begin{lemma}\cite[Lemma 2.1]{MP}\label{lemma-ind}
			Suppose  $\mathbb{X},\mathbb{Y}$ are finite-dimensional Banach spaces. If $\{x_1,x_2,\ldots,x_m\}$ is a linearly independent subset of $\mathbb{X}$ and $\{y_1^*,y_2^*,\ldots,y_n^*\}$  is a linearly independent subset of $\mathbb{Y}^*$ then $\{y_i^*\otimes x_j:1\leq i\leq n,1\leq j\leq m\}$ is a linearly independent subset of $\mathbb{L}(\mathbb{X},\mathbb{Y})^*.$  	
		\end{lemma}
		
		\begin{lemma}\label{lemma-2ngon}
			Let $\mathbb{X}=\ell_{\infty}^{n}$ and $\mathbb{Y}$ be a two-dimensional Banach space.  Let $T \in S_{\mathbb{L}(\mathbb{X}, \mathbb{Y})}$ be such that $Rank(T)=2$ and $Ext(B_\mathbb{X})\subseteq M_T.$ Then the followings hold:\\
			(i) $T(B_{\mathbb{X}})$ is a convex set with $4$ extreme points.\\
			 (ii) If $T(B_\mathbb{X})$ is the convex hull of $\{\pm z_1,\pm z_2\},$ then either for each  $x \in Ext(B_{\mathbb{X}}),$  $Tx\in \pm L[z_1,z_2]$ or for each  $x \in Ext(B_{\mathbb{X}}),$  $Tx\in \pm L[z_1,-z_2].$
		\end{lemma}
	\begin{proof}
		$(i)$ Follows from \cite[Remark  2.13]{MPD}.\\
		
		$(ii)$ Suppose $e_i=(0,0,\ldots,1,0,\ldots,0)$ with $1$ at the $i-$th coordinate and $0$ at the remaining places. Since, Rank$(T)=2,$ $Te_i \neq 0$ for some $i \in \{1,2,\dots,n \}.$  Without loss of generality, we assume $i=1,$ i.e., $Te_1 \neq 0.$ It is well-known that $B_\mathbb{X}=conv(K\cup-K),$ where $K=\{(1,u_2,\dots,u_n) : |u_i| \leq 1, 2 \leq i \leq n \}.$ Now, $K$ can be expressed as $K=e_1 + F,$ where $e_1=(1,0,\dots,0)$ and $F=\{(0,u_2,\dots,u_n) : |u_i| \leq 1, 2 \leq i \leq n \}.$  Observe that $x\in Ext(B_\mathbb{X})$ if and only if there exists $ u \in Ext(F)$ such that either $x= e_1 + u$ or $x=-e_1+u.$ Clearly, $T(B_{\mathbb{X}})=conv\Big(T(K)\cup T(-K)\Big),$ where $T(K)=T(e_1+F)=Te_1+T(F)$ and $T(-K)=T(-e_1+F)=-Te_1+T(F).$  Now, $T(F)$ must be a symmetric convex set about origin, since $F$ is so. If $T(F)$ has more than four extreme points then proceeding similarly as in \cite[Lemma 2.11]{MPD}, it can be shown that there exists $v \in Ext(B_{\mathbb{X}})$ such that $v \notin M_T,$ a contradiction. Thus, $T(F)$ has at most four extreme points.\\
		First suppose $T(F)$ has only two extreme points say $\pm y,$ i.e., $T(F)$ is the convex hull of $\{y,-y\}.$ Clearly, $T(B_{\mathbb{X}})$ is the convex hull of $\pm Te_1 \pm y.$ Now, for each $x\in Ext(F),Tx\in L[y,-y].$ Therefore, for each $z\in Ext(B_\mathbb{X}),$ $Tz \in \pm L[Te_1+y, Te_1-y]$ and hence we are done.\\
		Next, suppose $T(F)$ has four distinct extreme points say $\pm y_1, \pm y_2.$ We prove the rest of the lemma in the following two steps.\\
		
		\textbf{Step 1.} We claim that either  $T(B_\mathbb{X})=conv\{\pm (Te_1-y_2),\pm (Te_1+y_1)\}$ or $T(B_\mathbb{X})=conv\{\pm (Te_1+y_2),\pm (Te_1+y_1)\}.$
		
		\smallskip
		
		 Since $y_1,y_2\in Ext(T(F)),$ there exist $x_1,x_2\in Ext(F)$ such that $Tx_1=y_1,Tx_2=y_2.$ Now, $T(F)=conv\{\pm y_1,\pm y_2\}$ gives that $T(B_{\mathbb{X}})=conv\{\pm Te_1 \pm y_1, \pm Te_1 \pm y_2 \}.$ Note that $\pm Te_1\pm y_i=\pm Te_1\pm Tx_i,$ for $i=1,2$ and $\pm e_1\pm x_i\in Ext(B_\mathbb{X}).$ Therefore, $\|\pm Te_1\pm y_i\|=1,$ for $i=1,2.$ Since $\mathbb{Y}$ is two-dimensional and $\{y_1+y_2, y_1-y_2 \}$ is linearly independent, $Te_1=a(y_1+y_2)+b(y_1-y_2),$ where $a, b \in \mathbb{R}.$ We assert that either $a=0,b\neq 0$ or $a\neq 0,b=0.$ Clearly, $a$ and $b$ can not be simultaneously zero as $Te_1 \neq 0.$ If possible, suppose that $a\neq 0,b\neq 0.$ If $a > 0, b > 0,$ then 
		$$Te_1-y_1=\frac{2a}{2a+2b+1}(Te_1+y_2)+\frac{2b}{2a+2b+1}(Te_1-y_2)+\frac{1}{2a+2b+1}(-Te_1-y_1)$$
		and
		 $\frac{2a}{2a+2b+1},\frac{2b}{2a+2b+1},\frac{1}{2a+2b+1}\in (0,1).$ Moreover, we have, $\|Te_1+y_2\|=\|Te_1-y_2\|=\|Te_1+y_1\|=1.$ Using this, it can be easily observed that $\|Te_1-Tx_1\|= \|Te_1-y_1\|<1,$  which contradicts that $e_1-x_1\in Ext(B_{\mathbb{X}}) \subseteq M_T.$ Similarly, considering the other possible cases $a<0,b<0$ or $a<0,b>0$ or $a>0,b<0,$ we get a contradiction. This establishes our assertion. Assume that $a=0,b\neq0.$\\
	     Then we have
	     \begin{eqnarray}\label{eq-1}
	     Te_1-y_1=\frac{2b}{2b+1}(Te_1-y_2)+\frac{1}{2b+1}(-Te_1-y_1)~\text{and}
	     \end{eqnarray}
	      \begin{eqnarray}\label{eq-2}
	      -Te_1-y_2=\frac{1}{2b+1}(Te_1-y_2)+\frac{2b}{2b+1}(-Te_1-y_1).
	      \end{eqnarray}
		Thus, the only extreme points of $T(B_{\mathbb{X}})$ are $\pm (Te_1-y_2)$ and $\pm (Te_1+y_1),$ i.e, $T(B_\mathbb{X})=conv\{\pm (Te_1-y_2),\pm (Te_1+y_1)\}.$\\
		Assuming  $a\neq 0,b=0$ and  proceeding similarly, we can show that $T(B_\mathbb{X})=conv\{\pm (Te_1+y_2),\pm (Te_1+y_1)\}.$\\
		
		\textbf{Step 2.} Claim that either $Tz \in \pm L[Te_1+y_1, -Te_1+y_2]$ for each $z \in Ext(B_{\mathbb{X}})$ or $Tz \in \pm L[Te_1+y_1, -Te_1-y_2]$ for each $z \in Ext(B_{\mathbb{X}}).$
		
		\smallskip

		Suppose that  $T(B_\mathbb{X})=conv\{\pm (Te_1-y_2),\pm (Te_1+y_1)\}.$
		Let $z \in Ext(B_{\mathbb{X}})$ be arbitrary.  We  show that $Tz \in \pm L[Te_1+y_1, -Te_1+y_2].$ Now, there exists $x \in Ext(F)$ such that either $z=x + e_1$ or $z=x-e_1.$ First let $z=x+e_1.$ Now, if $Tx$ is an interior point of $T(F),$ then $Tz=Tx + Te_1$ will be an interior point of $T(B_{\mathbb{X}})$ and hence $||Tz|| < ||T||,$ a contradiction as $z \in Ext(B_{\mathbb{X}}).$ Thus, $Tx$ is on the boundary of $T(F),$ i.e., $Tx\in \pm L[y_1,y_2]\cup \pm L[y_1,-y_2].$  If $Tx \in \pm L[y_1, y_2],$ then clearly $Tz=Tx+Te_1 \in \pm L[Te_1+y_1, -Te_1+y_2]$ and we are done.
		
		 If possible, let $Tx \in L(y_1, -y_2),$ i.e., $Tx=(1-\lambda) y_1+ \lambda (-y_2), 0 < \lambda < 1.$  Then by Equation (\ref{eq-2}), we have
		 \begin{eqnarray*}
		 -Tx+Te_1&=&(1-\lambda) (Te_1-y_1)+ \lambda (Te_1+y_2)\\
		 &=&(1-\lambda) (Te_1-y_1)+\lambda \Big(\frac{1}{2b+1}(-Te_1+y_2)+\frac{2b}{2b+1}(Te_1+y_1)\Big)\\
		&=&(1-\lambda) (Te_1-y_1)+ \frac{\lambda}{2b+1}(-Te_1+y_2)+\frac{2b \lambda}{2b+1}(Te_1+y_1).
		 \end{eqnarray*}
		Since $0 < \lambda < 1,$ we have, $1-\lambda,\frac{\lambda}{2b+1},\frac{2b \lambda}{2b+1}\in (0,1).$ Moreover, we have, $\|Te_1-y_1\|=\|-Te_1+y_2\|=\|Te_1+y_1\|=1.$ Using this, it can be easily observed that $\|-Tx+Te_1\|<\|T\|,$ where $-x+e_1 \in Ext(B_{\mathbb{X}}),$ a contradiction.

		Similarly, if $Tx \in -L(y_1, -y_2),$ then we can show that $\|Tz\|=\|Tx+Te_1\|<\|T\|,$ where $z=x+e_1 \in Ext(B_{\mathbb{X}}),$ a contradiction. Therefore, we must have $Tx \in \pm L[y_1, y_2],$ i.e., $Tz=Tx+Te_1 \in \pm L[Te_1+y_1, -Te_1+y_2].$\\
		Now if $z=x-e_1,$ following the same line of arguments, we can show that $Tz \in \pm L[Te_1+y_1, -Te_1+y_2].$\\
		On the other hand, if we assume that $T(B_\mathbb{X})=conv\{\pm (Te_1+y_2),\pm (Te_1+y_1)\},$ then proceeding similarly, we can show that $Tz \in \pm L[Te_1+y_1, -Te_1-y_2]$ for each $z \in Ext(B_{\mathbb{X}}).$
		This completes the proof of the lemma.
	\end{proof}

		The following lemma is needed to prove the desired theorem. The proof of which is trivial and hence we omit the proof.
		
		\begin{lemma}\label{lemma-k-face}
			Any face of $\ell_{\infty}^{n}$ having exactly $2^{k}$ extreme points contains exactly $k+1$ linearly independent extreme points.
		\end{lemma}

		Now, we are ready to prove our desired result. We completely characterize $k-$smoothness of an operator defined between $\ell_{\infty}^n$ and any two-dimensional Banach space with the condition that the operator attains its norm at each element of $Ext(B_{\ell_\infty^n}).$ We solve the problem in the following two theorems. In the following  theorem, we consider the case in which  image of each extreme point of $B_{\ell_\infty^n}$  is smooth. 
		
		\begin{theorem}\label{th-result 1}
			Let $\mathbb{X}=\ell_{\infty}^{n}$ and $\mathbb{Y}$ be a two-dimensional Banach space. Let $T \in S_{{\mathbb{L}}(\mathbb{X},\mathbb{Y})}$ be such that $Ext(B_{\mathbb{X}}) \subseteq M_T$ and $Tx$ is smooth for all $x \in Ext(B_{\mathbb{X}}).$ Then the followings hold:\\
			(i) If $Rank(T)=1,$ then $T$ is $n-$smooth.\\
			(ii) Let $Rank(T)=2.$ If $Tx$ is an interior point of some line segment of $T(B_{\mathbb{X}})$ for some $x \in Ext(B_{\mathbb{X}}),$ then $T$ is $n-$smooth. Otherwise, $T$ is $(2n-2)-$smooth.
		\end{theorem}
		\begin{proof}
			Let us write $Ext(B_{\mathbb{X}})=\{\pm x_1, \pm x_2, \dots, \pm x_{2^{n-1}} \},$ where $\{x_1, x_2, \dots, x_{n} \}$ is linearly independent.\\
			(i) Suppose $Rank(T)=1.$ Then $Tx_i=\pm Tx_1$ for all $2\leq i\leq 2^{n-1}.$ Let $J(Tx_1)=\{y^*\}.$ Then for any $i \in \{1,2,\dots, 2^{n-1} \},$ either $J(Tx_i)= \{y^* \}$ or $J(Tx_i)=\{-y^* \}.$  Now, if $T$ is $k-$smooth, then 
			\begin{eqnarray*}
				k&=& dim ~span ~J(T)\\
				&=& dim~ span~ Ext ~J(T)\\
				&=& dim~ span ~\{y^*\otimes x_i : 1 \leq i \leq 2^{n-1} \}\\
				&=& dim~ span ~\{y^*\otimes x_i : 1 \leq i \leq n \}\\
				&=& n,
			\end{eqnarray*}
			as $\{y^*\otimes x_i : 1 \leq i \leq n \}$ is linearly independent by Lemma \ref{lemma-ind}. Hence, $T$ is $n-$smooth.\\
			
			(ii) Suppose $Rank(T)=2.$ Then by Lemma \ref{lemma-2ngon}, $T(B_\mathbb{X})$ is a convex set with four extreme points. Let $\pm y_1,\pm y_2$ be four distinct extreme points of $T(B_\mathbb{X}).$
			
			\smallskip
			
			  First suppose $Tx$ is an interior point of some line segment of $T(B_{\mathbb{X}})$ for some $x \in Ext(B_{\mathbb{X}}),$ i.e.,  $Tx_j\in L(y_1,y_2)$ for some $1\leq j\leq 2^{n-1}.$ Again by Lemma \ref{lemma-2ngon}, we get $Tx_i\in \pm L[y_1,y_2]$ for all $1\leq i\leq 2^{n-1}.$ Let $J(Tx_j)=\{y^*\}.$ Then it is clear that for any $i \in \{1,2,\dots, 2^{n-1} \},$ either $J(Tx_i)= \{y^* \}$ or $J(Tx_i)=\{-y^* \}.$  Then as in case (i) it is easy to show  that $T$ is $n-$smooth.
			  
			  \smallskip
			
			Next, suppose that  $Tx$ is not an interior point of any line segment of $T(B_{\mathbb{X}})$ for any $x \in Ext(B_{\mathbb{X}}).$ Then $Tx_i\notin L(\pm y_1,\pm y_2)$ for any $1\leq i\leq 2^{n-1}.$ Thus, $Tx_i\in \{\pm y_1,\pm y_2\}$ for all $1\leq i\leq 2^{n-1}.$ Since, $Rank(T)=2,$ without loss of generality, we may assume $Tx_1=y_1,Tx_2=y_2.$  Let $J(Tx_1)=\{z_1^*\}$ and $J(Tx_2)=\{z_2^*\}.$ Then for any $i \in \{1,2,\dots, 2^{n-1} \},$
			\begin{eqnarray*}
				&J(Tx_i)=& \{z_1^* \} ~or~ \{-z_1^* \} ~or~ ~\{z_2^* \}~ or~ \{-z_2^* \}.
			\end{eqnarray*}
			Let $S_1= \{x_i \in Ext(B_{\mathbb{X}}) : Tx_i=Tx_1 \}$ and $S_2= \{x_i \in Ext(B_{\mathbb{X}}) : Tx_i=Tx_2 \}.$ Thus, we have $S_1 \cap S_2 = \emptyset,$ $\pm S_1 \cup \pm S_2 = Ext(B_{\mathbb{X}})$ and $ |S_1 \cup S_2| =|S_1| + |S_2|= 2^{n-1}.$    	
			Therefore, for any $i \in \{1,2,\dots, 2^{n-1} \},$
			\begin{eqnarray*}
				J(Tx_i)&=& \{z_1^* \}, ~\text{if~}x_i \in S_1\\
				&=& \{z_2^* \},~\text{if~} x_i \in S_2.
			\end{eqnarray*}
			Now, it is clear that $S_1$ as well as $S_2$ cannot contain $n$ linearly independent vectors. For otherwise, we get $Rank(T)=1,$ a contradiction. Thus, maximal linearly independent subsets of $S_1$ and $S_2$ contain at most $n-1$ elements. Now, we show that $|S_1|= 2^{n-2}$ and $|S_2|= 2^{n-2}.$ If possible, let $|S_1|< 2^{n-2}.$ Then we must have $|S_2|> 2^{n-2}.$ Observe that $conv(S_2)$ is a face of $B_\mathbb{X}.$ Clearly, $Ext(conv(S_2))=S_2,$ i.e., $|Ext(conv(S_2))|=|S_2|>2^{n-2}.$ Hence, the face $conv(S_2)$ of $B_\mathbb{X}$ contains at least $2^{n-1}$ extreme points and hence by Lemma \ref{lemma-k-face}, $conv(S_2)$ contains at least $n$ linearly independent extreme points. Thus, $S_2$ contains at least $n$ linearly independent vectors. This gives that $Rank(T)=1,$ a contradiction. Therefore,  $|S_1|\geq 2^{n-2}.$ Similarly, $|S_2|\geq 2^{n-2}.$ Thus, $|S_1|=|S_2|= 2^{n-2},$ i.e., $|Ext(conv(S_1))|=|Ext(conv(S_2))|=2^{n-2}.$ Now, by Lemma \ref{lemma-k-face}, $S_1$ and  $S_2$ has exactly $n-1$ linearly independent vectors. Without loss of generality, suppose that the set of all linearly independent vectors of $S_1$ and $S_2$ are $\{u_1,u_2,\ldots,u_{n-1}\}$ and $\{u_n,u_{n+1},\ldots,u_{2n-2}\}$ respectively. Now, if $T$ is $k-$smooth, then 
			\begin{eqnarray*}
				k&=& \dim~span~J(T)\\
				&=& \dim~span~Ext ~J(T)\\
				&=& \dim~span~\{z_1^*\otimes u_i, z_2^*\otimes u_j : u_i \in S_1, u_j \in S_2 \}\\
				&=& \dim~span~\{z_1^*\otimes u_i, z_2^*\otimes u_j : 1 \leq i \leq n-1,n\leq j\leq 2n-2 \}\\
				&=&2n-2, \text{which can be easily verified}.
			\end{eqnarray*}
			Therefore, $T$ is $(2n-2)-$smooth. This completes the proof.
		\end{proof}
		
		In addition to Theorem \ref{th-result 1}, if we assume that the range space is strictly convex and smooth, then we obtain the following corollary.
		\begin{corollary}
			Let $\mathbb{X}=\ell_{\infty}^{n}$ and $\mathbb{Y}$ be a two-dimensional strictly convex, smooth Banach space. Let $T \in S_{{\mathbb{L}}(\mathbb{X},\mathbb{Y})}$ be such that $Ext(B_{\mathbb{X}}) \subseteq M_T.$ Then the followings hold:\\
			(i) If $Rank(T)=1,$  then $T$ is $n-$smooth.\\
			(ii)  If $Rank(T)=2,$  then $T$ is $(2n-2)-$smooth.
		\end{corollary}
		\begin{proof}
			$(i)$ follows clearly from Theorem \ref{th-result 1}. We only prove $(ii)$. Let $Rank(T)=2.$ Then by Lemma \ref{lemma-2ngon}, $T(B_\mathbb{X})$ is a convex set with four extreme points. Without loss of generality, let $\pm Tx_1,\pm Tx_2$ be four distinct extreme points of $T(B_\mathbb{X}).$ If possible, suppose that there exists $x\in Ext(B_{\mathbb{X}})$ such that $Tx$ is an interior point of some line segment of $T(B_\mathbb{X}).$ Suppose that $Tx\in L(Tx_1,Tx_2).$ Since $x\in M_T,\|Tx\|=1.$ Thus, it is clear that $\|y\|=1 $ for all $y\in L[Tx_1,Tx_2],$ i.e., $ L[Tx_1,Tx_2]\subseteq S_{\mathbb{Y}}.$ This contradicts that $\mathbb{Y}$ is strictly convex. Therefore,  there does not exist any $x\in Ext(B_{\mathbb{X}})$ such that $Tx$ is an interior point of some line segment of $T(B_\mathbb{X}).$ Hence, from Theorem \ref{th-result 1}, we conclude that if $Rank(T)=2,$  then  $T$ is $(2n-2)-$smooth.
		\end{proof}

		In the next theorem, we consider the remaining case in which the  image of at least one extreme point of $B_{\ell_\infty^n}$  is not a smooth point. Note  that in a two-dimensional Banach space, a non-zero vector is either smooth or $2-$smooth.

		\begin{theorem}\label{th-result 2}
			Let $\mathbb{X}=\ell_{\infty}^{n}$ and $\mathbb{Y}$ be a two-dimensional Banach space. Let $T \in S_{\mathbb{L}(\mathbb{X},\mathbb{Y})}$ be such that $Ext(B_{\mathbb{X}}) \subseteq M_T.$ Let $S=\{x\in Ext(B_\mathbb{X}):Tx~\text{is~not~smooth} \}$ be non-empty and $S_1$ be the subset of $S$ containing all linearly independent vectors of $S.$ If $|S_1|=k,$ then $T$ is $(n+k)-$smooth.
		\end{theorem}
		\begin{proof}
			Let $Ext(B_\mathbb{X})=\{\pm x_1,\pm x_2,\ldots,\pm x_{2^{n-1}}\}.$ First suppose that $Rank(T)=1.$ Since $S\neq\emptyset,$ assume $x_1\in S.$ Then for all $2\leq i\leq 2^{n-1},$ $Tx_i=\pm Tx_1.$ Let $Ext~J(Tx_1)=\{z_1^*,z_2^*\}$ for some $z_1^*,z_2^*\in S_{\mathbb{Y}^*}.$ Then either $Ext~J(Tx_i)=\{z_1^*,z_2^*\}$ or  $Ext~J(Tx_i)=\{-z_1^*,-z_2^*\}$ for all $2\leq i\leq 2^{n-1}.$ Clearly, $|S_1|=n.$  Now,
			\begin{eqnarray*}
				&& \dim ~span~Ext~J(T)\\
				&=& \dim ~span ~\{z_1^*\otimes x_i,z_2^*\otimes x_i:1\leq i\leq 2^{n-1} \}\\
				&=& \dim ~span ~\{z_1^*\otimes x_i,z_2^*\otimes x_i:1\leq i\leq n \}\\
				&=& 2n,~\text{using ~Lemma \ref{lemma-ind}}.
			\end{eqnarray*}
			Thus, in this case $T$ is $2n-$smooth and we are done. \\
			Next, suppose that $Rank(T)=2.$ Then by Lemma \ref{lemma-2ngon}, $T(B_\mathbb{X})$ is a convex set with four extreme points. Without loss of generality, let $\pm y_1,\pm y_2$ be four distinct extreme points of $T(B_\mathbb{X}).$  We consider the following  two cases:\\
		\textbf{Case I :} $S=Ext(B_\mathbb{X}).$\\
			Clearly, in this case $|S_1|=n.$ Let $S_1=\{x_1,x_2,\ldots,x_n\}.$ Observe that if $Tx_i\in L(\pm y_1,\pm y_2),$ for any $1\leq i\leq 2^{n-1},$ then $Tx_i$ will be smooth, which contradicts that $S= Ext(B_\mathbb{X}).$ Thus, $Tx_i\in \{\pm y_1,\pm y_2\}$ for all $1\leq i\leq 2^{n-1}.$ Since $Rank(T)=2,$ $Tx_i=y_1,Tx_j=y_2$ for some $1\leq i\neq j\leq 2^{n-1}.$ Therefore, $y_1,y_2$ are not smooth. Suppose that $Ext~J(y_1)=\{z_1^*,z_2^*\}$ and $Ext~J(y_2)=\{z_3^*,z_4^*\}.$ Then for each $1\leq i\leq 2^{n-1},$ $Ext~J(Tx_i)$ is either $\{z_1^*,z_2^*\}$ or $\{-z_1^*,-z_2^*\}$ or $\{z_3^*,z_4^*\}$ or $\{-z_3^*,-z_4^*\}.$ Observe that $z_3^*,z_4^*\in span\{z_1^*,z_2^*\},$ since $\dim(\mathbb{Y})=2.$  Hence, for any $x\in \mathbb{X},z_3^*\otimes x,z_4^*\otimes x\in span~\{z_1^*\otimes x,z_2^*\otimes x\}.$ Therefore, 
			\begin{eqnarray*}
				&& \dim ~span~Ext~J(T)\\
				&=& \dim ~span ~\{z_1^*\otimes x_i,z_2^*\otimes x_i:1\leq i\leq 2^{n-1} \}\\
				&=& \dim ~span ~\{z_1^*\otimes x_i,z_2^*\otimes x_i:1\leq i\leq n \}\\
				&=& 2n.
			\end{eqnarray*}
			Thus, in this case $T$ is $2n-$smooth and we are done. \\
			\textbf{Case II :}  $S\subsetneqq Ext(B_\mathbb{X}).$\\
			Without loss of generality, we may assume that $S_1=\{x_1,x_2,\ldots,x_k\}$ and $S=\{\pm x_1,\pm x_2,\ldots,\pm x_k,\pm x_{k+1},\ldots,\pm x_m\},$ $(1\leq m< 2^{n-1},1\leq k\leq n).$ As in Case I, $Tx_i\in \{\pm y_1,\pm y_2\}$ for all $1\leq i\leq m.$ Clearly, at least one of $y_1,y_2$ is not smooth. First we assume both of $y_1,y_2$ are not smooth. Using Lemma \ref{lemma-2ngon}, we get either $Tx_i\in \pm L[ y_1, y_2]$ for each $m < i\leq 2^{n-1}$ or $Tx_i\in \pm L[- y_1, y_2]$ for each $m <  i\leq 2^{n-1}.$ Without loss of generality, assume that $Tx_i\in \pm L[ y_1, y_2]$ for each $m <  i\leq 2^{n-1}$. Since, $y_1,y_2$ are not smooth and $Tx_i$ are smooth, $Tx_i\notin \{\pm y_1,\pm y_2\}$ for $m <  i\leq 2^{n-2}.$ Therefore, $Tx_i\in \pm L(y_1, y_2)$ for each $m <  i\leq 2^{n-1}$.  Now, it is easy to observe that either $J(Tx_i)=\{y^*\}$ or $J(Tx_i)=\{-y^*\}$ for each $m < i\leq 2^{n-1}.$ Suppose $Ext~J(y_1)=\{y^*,z_1^*\}$ and $Ext~J(y_2)=\{y^*,z_2^*\}.$ Then for each $1\leq i\leq m,$ $Ext~J(Tx_i)$ is either $\{y^*,z_1^*\}$ or $\{-y^*,-z_1^*\}$ or $\{y^*,z_2^*\}$ or $\{-y^*,-z_2^*\}.$ Observe that $z_2^*\in span\{y^*,z_1^*\},$ since $\dim(\mathbb{Y})=2.$ Hence, for any $x\in \mathbb{X},z_2^*\otimes x\in span~\{y^*\otimes x,z_1^*\otimes x\}.$ Therefore,
			\begin{eqnarray*}
				&& \dim ~span~Ext~J(T)\\
				&=& \dim ~span ~\{y^*\otimes x_i,z_1^*\otimes x_j:1\leq i\leq 2^{n-1},1\leq j\leq m \}\\
				&=& \dim ~span ~\{y^*\otimes x_i,z_1^*\otimes x_j:1\leq i\leq 2^{n-1},1\leq j\leq k \}\\
				&=& n+k, ~(\text{by~Lemma~\ref{lemma-ind}}),
			\end{eqnarray*}
			since $\{x_i:1\leq i\leq 2^{n-1}\}$ contains only $n$ linearly independent vectors. Therefore, $T$ is $(n+k)-$smooth. \\
			Now, if we consider exactly one of $y_1,y_2$ is not smooth, then following same line of arguments, we can prove that $T$ is $(n+k)-$smooth.
			This completes the proof of the theorem.		
		\end{proof}
		
		We would like to mention that Theorem \ref{th-result 1} and Theorem \ref{th-result 2} improves on  \cite[Th. 3.10]{MDP}. The study of $k-$smoothness of an operator defined between $ \ell_{\infty}^n$ and $\mathbb{Y}$ becomes more complicated when  $\dim\mathbb{Y} \geq 3.$ 
		The rest of the paper is devoted to the study of $k-$smoothness of an operator defined between two particular spaces   $ \ell_{\infty}^3 $ and $ \ell_1^3.$ We denote the extreme points of $B_{\ell_{\infty}^3}$ by $\pm x_1=\pm (1,1,1), \pm x_2=\pm (-1,1,1), \pm x_3=\pm (-1,-1,1), \pm x_4=\pm (1,-1,1). $  
		$|M_T \cap Ext(B_{\ell_{\infty}^3})|$ plays an important role in determining the order of smoothness of $T.$  Observe that if $|M_T \cap Ext(B_{\ell_{\infty}^3})| \leq 6,$ then the order of smoothness of $T$ can be obtained using \cite[Th. 2.2]{MP}. Therefore, we only consider the case for which $|M_T \cap Ext(B_{\ell_{\infty}^3})|=8,$ i.e., $M_T \cap Ext(B_{\ell_{\infty}^3})=\{\pm x_1, \pm x_2, \pm x_3, \pm x_4 \}.$ Note that, for $1 \leq i \leq 4,$ $Tx_i$ is $k-$smooth, where $k \in \{1,2,3 \}.$ Suppose $S_k=\{x \in M_T \cap Ext(B_{\ell_{\infty}^3)}: Tx~is~ k-smooth \},$ where $k \in \{1,2,3 \}.$ Clearly, $|S_1|+|S_2|+|S_3|=8.$ In the following theorem, we consider the case when $|S_1|=8.$
		
		\begin{theorem}\label{th-001a}
			Let $\mathbb{X}=\ell_{\infty}^3$ and $\mathbb{Y}=\ell_{1}^3.$ Let $T \in S_{\mathbb{L}(\mathbb{X}, \mathbb{Y})}$ be such that $M_T \cap Ext(B_{\mathbb{X}})= \{ \pm x_1, \pm x_2, \pm x_3, \pm x_4 \}.$ Let $|S_1|=8.$ Then the following hold:\\
			$(i)$  If $\pm J(Tx_i)= \pm J(Tx_j)$ for all $x_i,x_j \in S_1$, then $T$ is $3-$smooth.\\
			$(ii)$ Otherwise, $T$ is $4-$smooth.
		\end{theorem}
		
		\begin{proof}
			$(i)$  Suppose the given condition is satisfied. Let $\pm J(Tx_i)=\{\pm y^*\}$ for $ 1 \leq i \leq 4.$ Now, if $T$ is $k-$smooth, then 
			\begin{eqnarray*}
				k&=& \dim ~span ~J(T)\\
				&=& \dim~ span~ Ext ~J(T)\\
				&=& \dim~ span ~\{y^*\otimes x_1, y^*\otimes x_2, y^*\otimes x_3, y^*\otimes x_4\}\\
				&=& \dim~ span ~\{y^*\otimes x_1, y^*\otimes x_2, y^*\otimes x_3\}\\
				&=& 3, ~(\text{using~ Lemma~\ref{lemma-ind}}).
			\end{eqnarray*}
			Hence, $T$ is $3-$smooth.\\
			$(ii)$  Let $\pm J(Tx_i)=\{\pm y_i^*\}$ for $ 1 \leq i \leq 4.$ Since $(i)$ is not satisfied, without loss of generality, we assume $y_1^*\neq \pm y_2^*,$ i.e., $\{y_1^*,y_2^*\}$ is linearly independent.  Let $y_3^*=ay_1^*+by_2^*$ and $y_4^*=cy_1^*+dy_2^*,$ where $a,b,c,d\in \mathbb{R}.$ Since $\|y_3^*\|=1,$ $a$ and $b$ cannot be zero simultaneously. Similarly, $c$ and $d$ cannot be zero simultaneously.  Now, if $T$ is $k-$smooth, then
			\begin{eqnarray*}
				k&=&\dim ~span ~J(T)\\
				&=& \dim~ span~ Ext ~J(T)\\
				&=& \dim~ span ~\{y_1^*\otimes x_1, y_2^*\otimes x_2, y_3^*\otimes x_3, y_4^*\otimes x_4 \}.
			\end{eqnarray*}
		    Using the relations $x_4=x_1-x_2+x_3,$ $y_3^*=ay_1^*+by_2^*,$ $y_4^*=cy_1^*+dy_2^*$ and Lemma \ref{lemma-ind}, it is easy to verify that  
		 $\{y_1^*\otimes x_1, y_2^*\otimes x_2, y_3^*\otimes x_3, y_4^*\otimes x_4 \}$ is linearly independent. Therefore, $k=4$ and $T$ is $4-$smooth. 
		\end{proof}

	Proceeding similarly we can find the $k-$smoothness of operator $T$ for other feasible cases. We skip the details of proof to avoid monotonicity. 
		The following two tables  illustrates the various possible cases of $k-$smoothness under different conditions on $S_1,S_2$ and $S_3.$ The first table contains the cases when $S_3 = \emptyset ,$ i.e., $Tx_i$ is either $1-$smooth or $2-$smooth  but not $3-$smooth.

		\begin{center}
			\begin{tabular}{|c|c|c|c|c|}
				\hline 
				&                       &                     &                                        &                                         \\
				$|S_1|$ & $|S_2|$ & $|S_3|$ & Further conditions on the operator $T$& $T$ is \\
				                     &                       &                     &                                      &               $k$-smooth \\
				                      &                       &                         &                                     &                with $k=$ \\
				\hline
				\multirow{2}{*}{8}  & \multirow{2}{*}{0}& \multirow{2}{*}{0} & {$\pm J(Tx_i)=\pm J(Tx_j)~\forall~x_i,x_j\in S_1$}&{3} \\ \cline{4-4}\cline{5-5}&&&{Otherwise} & {4}\\
				\hline
				\multirow{2}{*}{6}  & \multirow{2}{*}{2}& \multirow{2}{*}{0} & {$\pm J(Tx_i)=\pm J(Tx_j)~\forall~x_i,x_j\in S_1$ and }&{}\\
				&&& { $\pm J(Tx_i)\subseteq \pm Ext~ J(Tx_k)$,$~\forall x_i\in S_1,x_k\in S_2$}& {4}\\
				 \cline{4-4}\cline{5-5}&&&{Otherwise} & {5}\\
				\hline
				\multirow{2}{*}{4}  & \multirow{2}{*}{4}& \multirow{2}{*}{0} & {$\pm J(Tx_i)=\pm J(Tx_j)~\forall~x_i,x_j\in S_1$ and }&{}\\
				&&& {$\pm J(Tx_i)\subseteq \pm Ext~J(Tx_k)$,$\forall x_i\in S_1, \forall x_k\in S_2$}&{5}\\
				\cline{4-4}\cline{5-5}&&&{Otherwise} & {6}\\
				\hline
				\multirow{2}{*}{2}  & \multirow{2}{*}{6}& \multirow{2}{*}{0} & {$|\cap_{x_k \in S_2}\pm J(Tx_k)|\geq 2$ and }&{}\\
				&&& { $\pm J(Tx_i)\subseteq \pm Ext~J(Tx_k)$,$\forall x_i\in S_1, \forall x_k\in S_2$}&{6}\\
				\cline{4-4}\cline{5-5}&&&{Otherwise} & {7}\\
				\hline
				\multirow{3}{*}{0}  & \multirow{3}{*}{8}& \multirow{3}{*}{0} & {$|\cap_{i=1}^4\pm J(Tx_i)|=4~$}&{6}\\
				\cline{4-4}\cline{5-5}&&&{Either $|\cap_{i=1}^4\pm J(Tx_i)|=2~$ or}&\\
				&&& {$|\pm Ext~J(Tx_i)\cap \pm Ext~J(Tx_j)|\neq 2,$ for $1\leq i\neq j\leq 4$}& {7}\\
				 \cline{4-4}\cline{5-5}&&&{Otherwise} & {8}\\
				 \hline
				\end{tabular}
			\end{center}
		The next table exhibits the cases when $ S_3 \neq \emptyset.$

			\begin{center}
			\begin{tabular}{|c|c|c|c|c|}
				\hline 
				&                       &                     &                                        &                                         \\
				$|S_1|$ & $|S_2|$ & $|S_3|$ & Further conditions on the operator $T$& $T$ is \\
				&                       &                     &                                      &               $k$-smooth \\
				&                       &                         &                                     &                with $k=$ \\
				\hline 
				 \multirow{2}{*}{6}  & \multirow{2}{*}{0}& \multirow{2}{*}{2} & {$\pm J(Tx_i)=\pm J(Tx_j), \forall~x_i,x_j\in S_1$}&{5} \\ \cline{4-4}\cline{5-5}&&&{Otherwise} & {6}\\
				  \hline
				 \multirow{2}{*}{4}  & \multirow{2}{*}{2}& \multirow{2}{*}{2} & {$\pm J(Tx_i)=\pm J(Tx_j), \forall~x_i,x_j\in S_1$ and }&{} \\
				 &&&{$\pm J(Tx_i)\subseteq \pm Ext~J(Tx_j)$, $\forall~x_i\in S_1,x_j\in S_2$}&{6}\\
				  \cline{4-4}\cline{5-5}&&&{Otherwise} & {7}\\
				\hline
				$2$ & $4$ & 2 &  - & 7\\
				\hline
				\multirow{2}{*}{0}  & \multirow{2}{*}{6}& \multirow{2}{*}{2} & {$|\cap_{x_i\in S_2}\pm Ext~ J(Tx_i)|=4$}&{7} \\
				\cline{4-4}\cline{5-5}&&&{Otherwise} & {8}\\
				\hline
				$4 $ & $0$ & $4$ & -   & $7$\\
				\hline
				$0$ & $4$ & $4$ & - & $8$ \\
				\hline
				$0$ & $0$ & $8$ & - & $9$   \\  
				\hline
			\end{tabular}
		\end{center}
	\vspace{.2cm}
		Finally we would like to note that the following possibilities are not feasible:
		(i) $|S_1|=2,|S_2|=2,|S_3|=4,$ (ii) $|S_1|=2,|S_3|=6$ and (ii) $|S_2|=2,|S_3|=6.$

		\bibliographystyle{amsplain}

	\end{document}